\documentclass[12 pt]{article}

\usepackage[utf8]{inputenc}
\usepackage{amsmath}
\usepackage{amsfonts}
\usepackage{amssymb}
\usepackage{mathtools}

\usepackage{cite}

\usepackage{graphicx}
\usepackage{caption}
\usepackage{listings}
\usepackage{verbatim}
\usepackage{wrapfig}
\usepackage{array}
\usepackage{subfigure}

\usepackage{tabularx}

\usepackage{caption}

\usepackage{amsthm}

\theoremstyle{lema}

\theoremstyle{proposition}
\newtheorem{proposition}{Proposition}[section]

\theoremstyle{theorem}
\newtheorem{theorem}{Theorem}[section]

\theoremstyle{theorem}
\newtheorem{remark}{Remark}[section]

\theoremstyle{corollary}
\newtheorem{corollary}{Corollary}[section]

\theoremstyle{definition}
\newtheorem{definition}{Definition}[section]

\theoremstyle{example}
\newtheorem{example}{Example}[section]

\def\r{\mathbb R}

\providecommand{\keywords}[1]
{
	\small	
	\textbf{\textit{Keywords---}} #1
}

\providecommand{\msc}[1]
{
	\small	
	\textbf{\textit{Mathematics Subject Classification---}} #1
}

\title{On Significant Pointwise Contractions and Annular Contractions}
\author{Ovidiu Popescu 
	\\ \small{Faculty of Mathematics and Computer Science,}\\ \small{Transilvania University of Bra\c sov, Bulevardul Eroilor 29, Bra\c sov}\\
	\small{email: {ovidiu.popescu@unitbv.ro}}}
\date{}

\begin{document}
	
	\maketitle

	\begin{abstract}
		In the current paper, we introduce the class of \textit{significant pointwise contractions}. This family of mappings extends the classical Banach contractions and also contains nonexpansive mappings. We prove that in a complete metric space every significant pointwise contraction admits a fixed point that can be approximated through Picard iterations. We provide examples showing that the newly introduced class differs from the class of piecewise contractions and is not contained in it. Moreover, we point out that significant pointwise contractions are continuous and may possess multiple fixed points. We further introduce the class of \textit{significant pointwise shrinkings}, which generalizes the class of shrinkings, and prove that any such mapping has a fixed point in a compact metric space. Moreover, we define the class of \textit{annular pointwise contractions}, which includes Banach contractions but is distinct from the class of Suzuki mappings. Finally, we obtain a fixed point theorem for annular pointwise contractions and present an additional condition ensuring the existence of fixed points.
	\end{abstract}
	
	
	\msc{47H10, 47H09}

	\section{Introduction}
	
	\noindent 
	
	In 1922, Banach \cite{Banach} proved the well-known Contraction Mapping Principle which states the following:
	
	\begin{theorem}
		If $(X,d)$ is a complete metric space and $T:X\to X$ is a contraction, denoted (C), i.e. there exists a $c\in [0,1)$ such that 
		\[ d(Tx,Ty) \leq c\cdot d(x,y),\]
		for every $x,y \in X$, then $T$ has a unique fixed point, that is, there exists a unique \(x^* \in X\) such that $Tx^*=x^*$.
	\end{theorem} 

	Many and various useful applications and generalizations of this theorem have been obtained e.g. Kannan \cite{Kannan}, Chatterjea \cite{Chatterjea}, Boyd-Wong \cite{BoydWong}, Reich \cite{Reich}, Ciri\'c \cite{Ciric, Ciric2}, Meir-Keeler \cite{Meir-Keeler}, Suzuki \cite{Suzuki}, Wardowski \cite{Wardowski}. In some of the generalizations, the assumption of contractiveness is relaxed to a local condition, see \cite{Clarke, Edelstein, Edelstein2, HuKirk, Jungck}.
	
	\begin{definition}
		A map $T: X \to X$ is called \textit{pointwise contractive}, denoted (PC), if for every point $x \in X$, there exists a \(c_x \in [0,1)\) and an open neighborhood of \(x\), \(U_x \subseteq X\), such that \(d(Tx,Ty) \leq c_x d(x,y)\), for all \(y \in U_x\). We say that $T$ is \textit{uniformly pointwise contractive}, denoted (uPC), provided that the same \(c \in [0,1)\) works for all \(x\in X\).
	\end{definition}

	In 1978, Hu and Kirk \cite{HuKirk} gave the following theorem, with a proof corrected by Jungck \cite{Jungck} in 1982.
	
	\begin{theorem}
		If \((X,d)\) is a rectifiably path connected complete metric space and a map \(T : X \to X\) is (uPC), then \(T\) has a unique fixed point.
	\end{theorem}

	Recall, that a metric space \((X,d)\) is \textit{rectifiably path connected} if for any two points \(x,y \in X\), there exists a continuous map \(p: [0,1] \to X\) satisfying \(p(0)=x\), \(p(1)=y\) and having a finite length $l(p)$ defined as the supremum over all numbers:
	\[ \sum_{i=1}^{n} d(p(t_i),p(t_{i-1})),\]
	where \(0=t_0<t_1<t_2< \dots < t_n =1.\)
	
	In \cite{Munkres}, Munkres provides an example of a (PC) map \(T:\r \to \r \), defined by \(Tx= \dfrac12 (x+\sqrt{x^2+1})\), for every \(x \in \r\), without fixed or even periodic points (i.e. \(T^{(n)} = T\circ T \circ \dots \circ T\) has no fixed points for every \(n \geq 1\)).
	
	In 2016, Ciesielski and Jasinski \cite{Cieselski} proved that the fixed point result with the weaker requirement of \(T\) being (PC) remains true when we additionally assume that \(X\) is compact.
	
	\begin{theorem}[see Theorem 1.4 from \cite{Cieselski}]
		Assume that \((X,d)\) is a compact and rectifiably path connected metric space. If \(T : X \to X \) is a (PC), then \(T\) has a unique fixed point. 
	\end{theorem}

	Moreover, they gave a similar result without the assumption that the space is rectifiably path connected.
	
	\begin{definition}[see Definition 1.5 from \cite{Cieselski}]
		A map \(T : X \to X \) is \textit{locally contractive}, denoted (LC), if for every \(x \in X\) there exist \(c_x \in [0,1)\) and $\varepsilon_x > 0$ such that the restriction of \(T\) to the open ball centred at \(x\) and of radius \(\varepsilon_x\) (\(B(x,\varepsilon_x)\)) is a contraction with constant \(c_x\). Moreover, \(T\) is \textit{uniformly locally contractive} (uLC), if the same $\lambda$ and $\varepsilon$ work for all $x \in X$, i.e. $T$ is (\(\varepsilon, \lambda\))-(uLC).
 	\end{definition}
		
	\begin{theorem}[see Theorem 1.6 from \cite{Cieselski}]
		Assume that \((X,d)\) is complete and \(T : X \to X \) is (uLC).
		\begin{itemize}
			\item[i)] If \(X\) is connected, then \(T\) has a unique fixed point.
			\item[ii)] If \(X\) has a finite number of components, then \(T\) has a periodic point.
		\end{itemize}
	\label{Th1.6}
	\end{theorem}

	The following examples show that the map \(T\) in Theorem \ref{Th1.6} needs to have neither periodic points, when \(X\) has an infinite number of components, nor fixed points when \(X\) is disconnected with a finite number of components.
	
	\begin{example}[see Example 1.7 from \cite{Cieselski}]
		Consider \[ X  = \underset{n \geq 1}{\cup} [n, n+2^{-(n+1)}]\] with the standard distance and define \(T : X \to X \) by \[Tx = n+1+\dfrac{x-n}2,\] for every \(n \geq 1\) and \(x \in [n, n+2^{-(n+1)}]\). Then \(X\) is complete, while \(T\) is a \((\frac12,\frac12)\)-(uLC) map with no periodic points.
	\end{example}

	\begin{example}[see Example 1.8 from \cite{Cieselski}]
		Consider \[ X  = [-2,-1]\cup[1,2]\] with the standard distance and define \(T : X \to X \) by \[Tx =-\dfrac{x}{|x|},\]  \(x \in X\). Then \(X\) is compact, while \(T\) is a (uLC) map with no fixed points.
	\end{example}

	However, for compact spaces the periodic points must exist, even when (uLC) is weakened to the following notions.
	
	\begin{definition}[see Definition 4.1 from \cite{Cieselski}]
		A map \(T : X \to X \) is \textit{locally shrinking}, denoted (LS), provided for every \(x \in X\) there exist $\varepsilon_x > 0$ such that the restriction of \(T\) to \(B(x,\varepsilon_x)\) is shrinking (S), i.e. \(d(Tu,Tv) < d(u,v)\), for every \(u,v \in B(x,\varepsilon_x)\). Moreover, \(T\) is \textit{uniformly locally shrinking} (uLS), if the same $\varepsilon$ works for all $x \in X$.
	\end{definition}

	In 1962, Edelstein \cite{Edelstein} gave the following result for (uLS):
	
	\begin{theorem}[see Theorem 2 from \cite{Edelstein}]
		If \((X,d)\) is compact and \(T : X \to X \) is (uLS), then \(T\) has a periodic point.
	\label{Th1.10}
	\end{theorem}

	Ciesielski and Jasinski \cite{Cieselski} showed that Theorem \ref{Th1.10} remains true when \(T\) is assumed to be only (LS).
	
	\begin{proposition}[see Proposition 4.3 from \cite{Cieselski}]
		(LS) implies (uLS) for maps \(T : X \to X \), with compact $X$.
	\end{proposition}

	\begin{corollary}[see Corollary 4.4 from \cite{Cieselski}]
		If \((X,d)\) is compact and \(T : X \to X \) is (LS), then \(T\) has a periodic point.
	\end{corollary}

	In 2008, Suzuki \cite{Suzuki} proved the following theorem:
	
	\begin{theorem}[see Theorem 2 from \cite{Suzuki}]
		Let \((X,d)\) be a complete metric space and \(T\) be a mapping on \(X\). Define a nonincreasing function $\theta$ from $[0,1)$ to $(\frac12,1]$ by
		\[ \theta(r) = \begin{cases}
			1, \qquad\qquad\quad  \text{if  } \quad 0 \leq r \leq (\sqrt{5}-1)/2,\\
				(1-r)r^{-2}, \quad \text{if  } \quad (\sqrt{5}-1)/2 \leq r \leq 2^{-1/2},\\
					(1+r)^{-1}, \;\;\quad \text{if  } \quad 2^{-1/2} \leq r < 1.\\
		\end{cases}\]
		Assume that there exists $r \in [0,1)$ such that 
		\[\theta(r) d(x,Tx) \leq d(x,y) \quad \text{implies} \quad d(Tx,Ty) \leq r d(x,y),\]
		for all \(x,y \in X\). Then, there exists a unique fixed point \(z\)  of \(T\). Moreover, 
		\[\lim\limits_{n\to \infty} T^nx = z,\]
		for all $x \in X$.
	\label{Th1.13}
	\end{theorem}

	In fact, in Theorem \ref{Th1.13}, the contractive condition holds for all $y \in X$, except $B(x, d(x,Tx))$ for $x \neq Tx$. We can say that the map $T$ in Theorem \ref{Th1.13} is \textit{uniform complementary pointwise contractive}, denoted (uCPC).
	
	In 2024, Pant \cite{Pant} extended the Banach Contraction Principle and defined a condition that applied to contraction mappings as well as nonexpansive mappings.
	
	\begin{theorem}[see Theorem 2.1 from \cite{Pant}]
		Let \((X,d)\) be a complete metric space and \(T : X \to X \) be such that for each \(x,y \in X\) with \(x \neq Tx\) or $y \neq Ty$, we have \[d(Tx,Ty) \leq c d(x,y),\] \(0\leq c <1\). Then \(T\) has a fixed point. \(T\) has a unique fixed point iff the above condition is satisfied for each $x \neq y$ in $X$,.
	\label{Th1.14}
	\end{theorem}

	Considering both the theoretical and application importance of the Banach contraction principle, the aim of this paper is to introduce two new classes of contractive operators that include Picard-Banach contractions and a particular case. These classes are similar to (uPC), respectively (uAPC) and larger then the class of contractive mappings in Theorem \ref{Th1.14}.
	
	We then study the existence and uniqueness of fixed points and prove some strong convergence theorems for Picard iteration used to approximate the fixed points. Moreover, we present examples to illustrate the generality of our new results.
	
	\section{Significant pointwise contractions}
	
	\begin{definition}
		Let \((X,d)\) be a metric space. A mapping \(T : X \to X \) is said to be \textit{significant pointwise contractive} (SPC) if  there exists \(c \in [0,1)\) such that 
		\begin{equation}
			d(x,y) \leq d(x,Tx) \quad \text{implies} \quad d(Tx,Ty) \leq c d(x,y), 
		\label{*}
		\end{equation}
		for all \(x,y \in X\).
	\end{definition}

	\begin{remark}
		It is obvious that (C) are (SPC) and that the contractive mappings in Pant's Theorem \ref{Th1.14} are (SPC).
	\end{remark}

	The following examples show that the class of (SPC) mappings is different from the class of (PC) and (uPC) mappings and it is not included in the class of (PC) and (uPC) mappings.
	
	\begin{example}
		Consider $X = [0,1]$ with the standard distance and define \(T : X \to X \) by \(Tx=x\) for every \(x \in X\). Then \(X\) is compact, while \(T\) is (SPC), but it is not (PC).
	\end{example}

	\begin{example}
		Consider $X = [0,1] \cup [2,3]$ with the standard distance and define \(T : X \to X \) by \(Tx=0\) if \(x \in [0,1]\) and \(Tx=1\) if \(x \in [2,3]\). Then \(T\) is (uPC), but it is not (SPC). Indeed, \(d(T1,T2) = 1 \geq d(1,2)\) and \(d(1,2) = 1 \leq d(1,T1)\).
	\end{example}

	\begin{theorem}
		Every (SPC) is continuous.
	\end{theorem}

	\begin{proof}
		Let \(x_0\) be a point in \(X\).  If \(x_0\) is isolated, then, clearly, \(T\) is continuous at \(x_0\). 
		
		Let now \(x_0\) be an accumulation point. Then there exists a sequence \(\{x_n\}\) in \(X\) such that \(\{x_n\} \to x\). 
		
		If \(x \neq Tx\), for \(n\) sufficiently large we have \(d(x_n,x) \leq d(x,Tx)\), so by (\ref{*}) we obtain \(d(Tx_n,Tx) \leq d(x_n,x)\). Taking the limit as \( n \to \infty \), we get \( d(Tx_n,Tx) \to 0\), thus \( T\) is continuous at \(x\).
		
		If \( x= Tx \), we suppose that \(T\) is not continuous at \(x\). Then, there exist \(\varepsilon > 0\) and a subsequence \(\{x_{n(k)}\}\) of \(\{x_n\}\) such that \(d(Tx_{n(k)}, x) > \varepsilon\), for every \(k \geq 1\). Since \(x_{n(k)} \to x\), there exists \(k_0 \geq 1\) such that \(d(x_{n(k)},x) < \frac{\varepsilon}{2}\), for \(k \geq k_0\). Then, we have 
		\[d(x_{n(k)}, Tx_{n(k)}) \geq d(Tx_{n(k)},x) - d(x_{n(k)},x)> \frac{\varepsilon}2 > d(x_{n(k)},x).\] By (\ref{*}), we obtain \(d(Tx_{n(k)},Tx) \leq c d(x_{n(k)},x)\).
		
		Then, 
		\[0< \varepsilon < c d(x_{n(k)},x) < c \frac{\varepsilon}{2}.\]
		
		Since \(c <1\), we have a contradiction. Therefore, \(T\) is continuous at \(x\).
	\end{proof}

	\begin{theorem}
		Let \((X,d)\) be a complete metric space and \(T : X \to X \) be an (SPC) mapping. Then, \(T\) has a fixed point.
		\label{Th2.25}
	\end{theorem}

	\begin{proof}
		Let \(x_0 \in X\) be arbitrarily chosen and let \(\{x_n\}_{n=0}^{\infty}\) be the Picard iteration, i.e. the sequence defined by \(x_n = Tx_{n-1} = T^nx_0\). 
		
		Since \(d(x_{n-1},x_n) = d(x_{n-1}, Tx_{n-1})\) for all \(n \geq 1\), taking \(x=x_{n-1}\) and \(y=x_n\) in (\ref{*}) we get 
		\[d(Tx_{n-1},Tx_n) \leq c d(x_{n-1},x_n).\]
		Hence, \(d(x_n,x_{n+1}) \leq c d(x_{n-1},x_n)\), for every \(n \geq 1\). This implies that 
		\[d(x_{n},x_{n+1}) \leq c^n d(x_{0},x_1),\]
		and then, for \(p \geq 1\), we have
		\begin{align*}
			d(x_n, x_{n+p}) &\leq \sum_{i=1}^p d(x_{n+i-1},x_{n+i}) \leq \sum_{i=1}^p c^{n+i-1} d(x_{0},x_{1})\\
			& = c^n (1+c+\dots+c^{p-1}) d(x_0,x_1) \leq \dfrac{c^n}{1-c} d(x_0,x_1).
		\end{align*}
		Since \(c^n \to 0\) as \(n \to \infty\), we obtain that \(\{x_n\}\) is a Cauchy sequence and, hence, it is convergent in the complete metric space \(X\), i.e. there exists \(x^* \in X\) such that 
		\[x^* = \lim\limits_{n\to \infty} x_n.\]
		
		Suppose that \(d(x^*,Tx^*) > 0.\) Then, for \(n\) sufficiently large we have \(d(x_n,x^*)\leq d(x^*,Tx^*)\), and by (\ref{*}) taking \(x=x^*\) and \(y=x_n\), we obtain \(d(Tx^*,Tx_n)\leq cd(x^*,x_n)\). Taking the limit as \(n \to \infty\), we get \(d(Tx^*,x^*) = 0\), which is a contradiction. Therefore, \(d(Tx^*,x^*) = 0\), so \(x^*\) is a fixed point of \(T\).
	\end{proof}
	
	\begin{remark}
		Additionally, if an (SPC) satisfies 
		\[d(x,y) \geq L d(x,Tx) \quad \text{implies} \quad d(Tx,Ty) \leq c d(x,y),\]
		for all \(x,y \in X\), where \(L > 1\), then \(T\) has a unique fixed point.
	\end{remark}

	\begin{example}
		Consider $X = \{(x,2k) : x \in [0,1],\, k \in \{1,2,\dots,n\}\}$ and \(d\) be the Euclidean metric. Then \(T:X \to X\) defined by \(T(x,2k) = \left(\frac{x}2,2k \right)\) is an (SPC) with \(c=\dfrac12\) and has \(n\) fixed points. Moreover, \(T\) does not not verify the hypothesis of Theorem \ref{Th1.14}.
	\end{example}

	\begin{example}[see Example 2.4. from \cite{Pant}]
		Consider $X = \{re^{i\theta} : 0 \leq \theta \leq 2\pi,\, r = 1,3,3^2,\dots\}$ be the self-similar family of concentric circles, each lying within larger circles having radii in geometric progression, in the \(xy\)-plane. Let \(d\) be the Euclidean metric. Define \(T:X \to X\) by \(Tz=\dfrac{z}{|z|}=\dfrac{z}{r}\). Then, \(T\) satisfies (\ref{*}) with \(c= \dfrac{1}{2}\) so $T$ is an (SPC) and each point of the unit circle \(z=e^{i\theta}\) is a fixed point while every other point is an eventually fixed point of \(T\) (i.e. there exists a positive integer \(N\) such that \(T^{n+1}x=T^nx\) for \(n\geq N\)).
	\end{example}

	\begin{definition}
		Let \((X,d)\) be a metric space. We say that \(T : X \to X \) is \textit{significant pointwise shrinking} (SPS), provided 
		\begin{equation}
			0<d(y,x)\leq d(x,Tx) \quad \text{implies} \quad d(Tx,Ty) < d(x,y),
			\label{**}
		\end{equation}
		for all \(x,y \in X\).
	\end{definition}

	\begin{theorem} 	
		Let \((X,d)\) be a compact metric space and \(T : X \to X \) be an (SPS) mapping. Then, \(T\) has a fixed point.
	\end{theorem}

	\begin{proof}
		For \(x\neq Tx\) we have \(0<d(x,Tx)\leq d(x,Tx)\) so by (\ref{**}) we have \(d(Tx,T^2x) < d(x,Tx)\). Let \(\beta = \inf \{d(x,Tx): x \in X\}.\)
		Since \(X\) is compact, there exists a sequence \(\{x_n\}_{n=0}^{\infty}\) in \(X\) such that \(x_n \to u \in X\), \(Tx_n \to v \in X\) and \(\beta = \lim\limits_{n\to \infty} d(x_n,Tx_n) = d (u,v)\). 
		
		Suppose \(\beta > 0\). If there exists a subsequence \(\{x_{n(k)}\}\) of \({x_n}\) such that \(x_{n(k)}\neq u\) for all \(k \geq 1\), we have \(0<d(x_{n(k)},u) \leq d(x_{n(k)},Tx_{n(k)})\) for \(k\) sufficiently large. By (\ref{**}) we get \(d(Tx_{n(k)}, Tu) < d(x_{n(k)},u)\)  . Taking the limit as \(k \to \infty\), we have \(d(v,Tu) \leq 0\), so \(v=Tu\). 
		
		Otherwise, we have \(x_n=u\), for \(n\) sufficiently large, by where \(Tx_n = Tu\) and \(\beta = \lim\limits_{n\to \infty} d(x_n,Tx_n) = d (u,Tu)\). But \(d(Tu,T^2u) < d(u,Tu) = \beta,\) which is a contradiction.
		
		Therefore, \(\beta = 0\), hence \(u=v\). 
		
		Suppose \(d(u,Tu) > 0\). Then, for \(n\) sufficiently large \(d(u,x_n) \leq d(u,Tu)\). By (\ref{**}), we obtain that \(d(Tu,Tx_n) < d(u,x_n)\). Taking the limit as \(n \to \infty\), we have \(d(Tu,u) = 0\), which is a contradiction. Hence, \(d(u,Tu)=0\), i.e. \(u\) is a fixed point of \(T\).
	\end{proof}

	\begin{example}
		Consider $X = [0,1] \cup [2,3]$ with the standard distance. Let \(T : X \to X \) be defined by \(Tx=1\) if \(x \in [0,1]\) and \(Tx=2\) if \(x \in [2,3]\). Then, \(X\) is compact and \(T\) is (SPS).
		
		Indeed, for \(x,y \in [0,1]\), or \(x,y \in [2,3]\), \(x\neq y\) we have\(d(Tx,Ty) = 0 < d(x,y)\). If \(x \in [0,1]\) and \(y \in [2,3]\), we get \(d(Tx,Ty) = 1\) and \(d(x,y) \geq 1\), with equality when \(x=1\) and \(y=2\). However, in this case we have \(d(x,Tx) =0\), \(d(y,Ty) = 0\), so \(d(x,y) > d(x,Tx)\) and \(d(y,x) > d(y,Ty)\). Hence, \(T\) satisfies (\ref{**}). Moreover, we have \(d(T1,T2) = d(1,2)\), so \(T\) is not (C).
	\end{example}

	\begin{remark}
		Additionally, if an (SPS) satisfies
		\[d(x,y) \geq L d(x,Tx) \quad \text{and} \quad x\neq y \quad \text{implies} \quad d(Tx,Ty) \leq c d(x,y),\]
		for all \(x,y \in X\), where \(L > 1\), then \(T\) has a unique fixed point.
	\end{remark}
	
	\section{Annular pointwise contraction}
	
	\begin{definition}
		Let \((X,d)\) be a metric space. A mapping \(T : X \to X \) is said to be \textit{annular pointwise contractive} (APC) if  there exists \(c \in [0,1)\) such that 
		\begin{equation}
			\dfrac{1}{1+c}d(x,Tx) \leq d(x,y)\leq \dfrac{1}{1-c}d(x,Tx) \quad \text{implies} \quad d(Tx,Ty) \leq c d(x,y), 
			\label{***}
		\end{equation}
		for all \(x,y \in X\).
	\end{definition}
	
	\begin{remark}
		It is obvious that (C) are (APC) and that the contractive mappings in Pant's Theorem \ref{Th1.14} are (APC).
	\end{remark}
	
	\begin{example}
		Consider $X = \{0,1,2,3\}$ with the standard distance. Let \(T : X \to X \) be defined by \(T0=T2= 0\) and \(T1=T3=1\). Then \(T\) is (APC), but \(T\) is not (SPC) and \(T\) does not satisfy the hypothesis in Pant's Theorem \ref{Th1.14}.
		
		Indeed, for \(x=0\) or \(x=1\), (\ref{***}) is obvious. If \(x=2\), since \(\frac23 d(2,T2) \leq d(2,y) \leq 2d(2,T2)\), we get \(y=0\), so \(d(T2,T0)=0 < \frac12 d(2,0) = 1\). For \(x=3\), since \(\frac23 d(3,T3) \leq d(3,y) \leq 2d(3,T3)\), we get \(y=0\) or \(y=0\). But \(d(T3,T1)=0 < \frac12 d(3,1) = 1\). Moreover, \(d(1,2) = 1 \leq d(2,T2) = 2\) and \(d(T1,T2)=1=d(1,2)\), so \(T\) is not (SPC) and does not satisfy the hypothesis in Pant's Theorem \ref{Th1.14}. Also, since \(\dfrac{1}{1+c}d(0,T0) = 0 \leq d(0,1)\) and \(d(T0,T1) > cd(0,1)\), \(T\) does not verify the hypothesis of Suzuki's theorem.

	\end{example}

	\begin{theorem}
		Let \((X,d)\) be a complete metric space and \(T : X \to X \) be a (APC) mapping. Then, \(T\) has a fixed point.
	\end{theorem}
	
	\begin{proof}
		Let \(x_0 \in X\) be arbitrarily chosen and let \(\{x_n\}_{n=0}^{\infty}\) be the Picard iteration, i.e. the sequence defined by \(x_n = Tx_{n-1} = T^nx_0\). 
		
		Since \(\dfrac{d(x_{n-1},x_n)}{1+c} = d(x_{n-1}, Tx_{n-1}) \leq \dfrac{d(x_{n-1},x_n)}{1-c}\) for all \(n \geq 1\), taking \(x=x_{n-1}\) and \(y=x_n\) in \ref{***} we get 
		\[d(Tx_{n-1},Tx_n) \leq c d(x_{n+1},x_n).\]
		Hence, \(d(x_n,x_{n+1}) \leq c d(x_{n-1},x_n)\), for every \(n \geq 1\). Like in the proof of Theorem \ref{Th2.25}, we obtain that \(\{x_n\}\) is a Cauchy sequence. Since \(X\) is complete we have \(x_n \to x^* \in X\).
		
		If there exists \(n\geq 1\) such that 
		\[\dfrac1{1+c}d(x_{n-1},x_n) > d(x_{n-1},x^*) \quad \text{and} \quad \dfrac1{1+c}d(x_{n},x_{n+1}) > d(x_{n},x^*),\]
		then,
		\begin{align*}
			d(x_{n-1},x_n) &\leq d(x_{n-1},x^*)+d(x_{n},x^*) < \dfrac1{1+c}[d(x_{n-1},x_n)+d(x_{n},x_{n+1})]\\
			&\leq \dfrac1{1+c}[d(x_{n-1},x_n)+cd(x_{n-1},x_n)] = d(x_{n-1},x_n),
		\end{align*}
		which is a contradiction. Therefore, we have
		\[\dfrac{1}{1+c}d(x_{n-1},x_n) \leq d(x_{n-1},x^*) \quad\text{or} \quad \dfrac1{1+c}d(x_{n},x_{n+1})< d(x_{n},x^*).\]
		
		Since 
		\begin{align*}
			d(x_{n+p}, x_{n}) &\leq \sum_{i=0}^{p-1} d(x_{n+i+1},x_{n+i}) \leq \sum_{i=0}^{p-1} c^{i} d(x_{n+1},x_{n})\\
			& =(1+c+\dots+c^{p-1}) d(x_{n+1},x_{n}) < \dfrac{1}{1-c} d(x_{n+1},x_{n}),
		\end{align*}
		taking the limit as \(p \to \infty\), we get \(d(x^*,x_n) \leq \dfrac{1}{1-c} d(x_{n+1},x_n)\), for all \(n \geq 0\).
		
		Therefore, there exists a subsequence \(\{x_{n(k)}\}\) of \(\{x_n\}\) such that 
		\[\dfrac1{1+c}d(x_{n(k)},x_{n(k)}+1) \leq d(x_{n(k)},x^*) \leq \dfrac{1}{1-c}d(x_{n(k)},x_{n(k)}+1),\]
		for every \(k \geq 1\). Then \(d(Tx_{n(k)},Tx^*) \leq c d(x_{n(k)},x^*)\), by where we get, as \(k\to \infty\), that \(d(x^*, Tx^*) = 0\), i.e. \(x^*\) is a fixed point of \(T\).
	\end{proof}

	\begin{remark}
		If, additionally, an (APC) satisfies the condition
		\[d(x,y) \geq L d(x,Tx) \quad \text{implies} \quad d(Tx,Ty) \leq c d(x,y),\]
		for all \(x,y \in X\), where \(L \geq \dfrac{1}{1-c}\), then \(T\) has a unique fixed point.
	\end{remark}
	
	\section{Conclusions}
	
	In the end, let us summarize the main contributions of the current paper.
	
	\begin{enumerate}
		\item We introduce the class of significant pointwise contractions, which includes Picard-Banach contractions and also contain nonexpansive mappings.
		\item We show that in a complete metric space, any significant pointwise contraction has a fixed point that can be approximated by means of Picard iterations.
		\item We present relevant examples to show that the class of significant pointwise contractions is different from the class of piecewise contractions and is not included in it.
		\item It is worth mentioning that any significant pointwise contraction is continuous and can have many fixed points.
		\item We introduce the class of significant pointwise shrinkings, which includes the class of shrinkings, and show that any significant pointwise shrinking has a fixed point in a compact metric space.
		\item We also introduce the class of annular pointwise contractions, which includes Picard-Banach contractions and is different from the class of Suzuki mappings.
		\item Last but not least, we obtain a fixed point theorem for annular pointwise contractions and give an additional condition under which we have a fixed point result.
		\item Moreover, the results contained in the present paper open the way to new interesting results and generalizations of Kannan contractions (see \cite{Popescu1}) or Chatterjea contractions (see \cite{Popescu2}), and many other results.
	\end{enumerate}

	\medskip
	\vspace{1.2ex}
	
\end{document}